\documentclass[amssymb,amstex,amsmath,11pt]{article}
\usepackage[margin=1.4in]{geometry}
\usepackage{amssymb,amsmath}              
\usepackage{amsmath,amssymb,latexsym,amsfonts,url}
\usepackage{graphicx}
\usepackage{amsthm}
\usepackage{enumitem}
\usepackage{amssymb}
\usepackage{color}
\usepackage{epstopdf}
\usepackage{cite}
\renewcommand{\epsilon}{\varepsilon}
\theoremstyle{plain}
\newtheorem{thm}{Theorem}[section]

\newtheorem{cor}[thm]{Corollary}

\newtheorem*{theorem*}{Theorem}
\newtheorem*{proposition*}{Proposition}
\theoremstyle{definition}

\theoremstyle{remark}

\def\({\left(}
\def\){\right)}

\begin{document}
\begin{title}
{Sharp lower bounds for the first eigenvalue of Steklov-type eigenvalue problems on a compact surface}
\end{title}
\begin{author}{Gunhee Cho \and Keomkyo Seo}\end{author}

\date{\today}

\maketitle

\begin{abstract}
\noindent Let $\Omega$ be a compact surface with smooth boundary and the geodesic curvature $k_g \ge {c > 0}$ along $\partial \Omega$ for some constant $c \in \mathbb{R}$. We prove that, if the Gaussian curvature satisfies $K \ge -\alpha$ for a constant $\alpha \ge 0$, then the first eigenvalue $\sigma_1$ of the Steklov-type eigenvalue problem satisfies
\[
\sigma_1 + \frac{\alpha}{\sigma_1} \ge c.
\]
Moreover, equality holds if and only if $\Omega$ is a Euclidean disk of radius $\frac{1}{c}$  and $\alpha = 0$. Furthermore, we obtain a sharp lower bound for the first eigenvalue of the fourth-order Steklov-type eigenvalue problem on $\Omega$. \\


\end{abstract}

\section{Introduction}

\noindent Let $(\Omega, g)$ be an $n$-dimensional compact Riemannian manifold with smooth boundary $\partial\Omega$. Consider the following classical Steklov problem:
\[
\begin{cases}
	\Delta u = 0 & \text{in } \Omega,\\
	\frac{\partial u}{\partial \nu} = \sigma\,u & \text{on } \partial\Omega,
\end{cases}
\]
where $\nu$ is the outward unit normal to $\partial\Omega$. It is well known that the spectrum of this problem is discrete and unbounded:
\[
0 = \sigma_0 < \sigma_1 \le \sigma_2 \le \cdots \longrightarrow \infty,
\]
each eigenvalue being repeated according to its multiplicity. In recent years, the Steklov problem has attracted extensive interest (see \cite{AM, DK, Escobar1997, Escobar1999, Escobar2000, FS2011, FS2016, FS2020, LS, PP, XX}, for instance). We refer readers to the surveys by Colbois–Girouard–Gordon–Sher \cite{CGGS} and Girouard–Polterovich \cite{GP} for a summary of recent progress in this area.

For a convex domain $\Omega$ in $\mathbb R^2$, Payne\cite{Payne1970} obtained a sharp lower bound for the first Steklov eigenvalue, which is given by the minimum of the curvature of the boundary of $\Omega$.  More precisely, he proved the following.

\begin{thm}[Payne \cite{Payne1970}] \label{thm: Payne}
Let $\Omega\subset\mathbb R^2$ be a bounded convex domain with smooth boundary $\partial\Omega$. Assume that the geodesic curvature $k_g \ge c>0$ along $\partial\Omega$ for some constant $c\in \mathbb{R}$. Then the first Steklov eigenvalue $\sigma_1$  satisfies
$$\sigma_1 \ge c \, .$$
Moreover, equality holds if and only if $\Omega$ is a disk of radius $\frac{1}{c}$.
\end{thm}

\noindent On the other hand, Weinstock\cite{Weinstock} obtained an upper bound for the first Steklov eigenvalue on a simply connected bounded domain $\Omega \subset \mathbb R^2$ as follows:
\[
\sigma_1 L(\partial\Omega) \le 2\pi,
\]
where $L(\partial\Omega)$ denotes the length of $\partial\Omega$. Moreover, equality holds if and only if $\Omega$ is a disk. Applying the Gauss-Bonnet theorem yields
\[
\sigma_1 L(\partial\Omega) \le 2\pi = \int_{\partial\Omega}k_g\,ds \le \left(\max_{\partial \Omega}k_g\right)  L(\partial\Omega),
\]
which implies that
\[
\sigma_1 \le \max_{\partial \Omega}k_g .
\]
 Thus, it follows from Theorem \ref{thm: Payne} that, for any convex planar domain, its first Steklov eigenvalue lies between the maximum and minimum of the geodesic curvature along the boundary.

In~\cite{Escobar1997}, Escobar extended Theorem~\ref{thm: Payne} to two-dimensional Riemannian manifolds. Specifically, he proved the following:

\begin{thm}[Escobar~\cite{Escobar1997}] \label{thm: Escobar}
	Let $(\Omega, g)$ be a compact surface with smooth boundary $\partial\Omega$ and with the Gaussian curvature satisfies $K \ge 0$. Assume that the geodesic curvature satisfies $k_g \ge {{c > 0}}$ along $\partial\Omega$ for some constant $c \in \mathbb{R}$. Then the first Steklov eigenvalue $\sigma_1$ satisfies
	\[
	\sigma_1 \ge c.
	\]
	Moreover, equality holds if and only if $\Omega$ is a Euclidean disk of radius $\frac{1}{c}$.
\end{thm}

\noindent Motivated by this, Escobar proposed the conjecture that if an $n$-dimensional Riemannian manifold $M^n$ $(n \ge 3)$ has nonnegative Ricci curvature, and the second fundamental form $A$ satisfies $A \ge cI$ on $\partial M$ for some positive constant $c$, then $\sigma_1 \ge c$ with equality holding only when $M$ is a Euclidean ball of radius $\frac{1}{c}$. This conjecture was recently resolved in the case of nonnegative sectional curvature by Xia--Xiong~\cite{XX}.

In this paper, we are concerned with  obtaining a sharp lower bound for the first (nonzero) eigenvalue of a Steklov-type eigenvalue problem on a compact surface. Let $\Omega$ be a compact surface with smooth boundary $\partial\Omega$, and let $\varphi \in C^\infty(\mathbb{R})$ satisfy
\begin{align}\label{condition on phi}
	\begin{cases}
		\dfrac{\varphi (t)}{t} \ge 0 & \text{for } t \in \mathbb{R} \setminus \{0\}, \\[6pt]
		{\displaystyle \lim_{t \to 0} \frac{\varphi (t)}{t}} & \text{exists}.
	\end{cases}
\end{align}
Consider the following Steklov-type eigenvalue problem:
\begin{align}\label{problem: Second}
	\begin{cases}
		\Delta u = \varphi(u) & \text{in } \Omega, \\[6pt]
		\dfrac{\partial u}{\partial\nu} = \sigma\,u & \text{on } \partial\Omega,
	\end{cases}
\end{align}
where $\nu$ denotes the outward unit {{conormal}} vector along $\partial\Omega$. Let $u(x)$ be the first eigenfunction with corresponding eigenvalue $\sigma_1$ of problem~\eqref{problem: Second}. If $\Omega$ is a compact surface whose Gaussian curvature satisfies $K(x) \ge -\varphi'(u(x))$ for all $x \in \Omega$, and the geodesic curvature satisfies $k_g \ge c > 0$ along $\partial \Omega$ for some constant $c \in \mathbb{R}$, then
\[
\sigma_1 + \frac{1}{\sigma_1} \max_{\partial\Omega} \frac{\varphi(u(x))}{u(x)} \ge c.
\]
Moreover, equality holds if and only if $\Omega$ is a Euclidean disk of radius $\frac{1}{c}$ and $\varphi \equiv 0$. This result is one direction of extension of Theorems~\ref{thm: Payne} and~\ref{thm: Escobar} (see Theorem~\ref{thm: main1}). In particular, when $\varphi(t) = \alpha t$ for some constant $\alpha \ge 0$, the inequality reduces to
\[
\sigma_1 + \frac{\alpha}{\sigma_1} \ge c.
\]
In this case, equality holds if and only if $\Omega$ is a Euclidean disk of radius $\frac{1}{c}$ and $\alpha = 0$ (see Corollary~\ref{cor: Schrodinger}).

In Section 3, we study a fourth-order Steklov-type eigenvalue problem. Let $\Omega$ be a compact surface with smooth boundary $\partial \Omega$. Consider the fourth-order eigenvalue boundary problem on $\Omega$ as follows:

\begin{align} \label{problem: Fourth}
\begin{cases}
	\Delta^2 u = \varphi(u) & \text{in } \Omega, \\
	u = 0 & \text{on } \partial \Omega, \\
	\Delta u = p_1 \dfrac{\partial u}{\partial \nu} & \text{on } \partial \Omega,
\end{cases}
\end{align}
where $\nu$ denotes the outward unit conormal vector along $\partial \Omega$ and $\varphi \in C^\infty(\mathbb{R})$. In particular, when $\varphi=0$, this problem was initiated by Kuttler-Sigillito \cite{KS} and by Payne \cite{Payne1970}. Since then, it has been developed by many mathematicians (see \cite{BGM, BFG, FGW, RS} for example). In \cite{Payne1970}, Payne obtained a sharp lower bound for the first eigenvalue of the problem (\ref{problem: Fourth}) if $\Omega \subset \mathbb{R}^2$ is a bounded convex domain and $\varphi = 0$. More precisely,

\begin{thm}[Payne~\cite{Payne1970}] \label{thm: Payne fourth}
	Let $\Omega \subset \mathbb{R}^2$ be a bounded convex domain with smooth boundary $\partial\Omega$. Assume that the geodesic curvature satisfies $k_g \ge c > 0$ along $\partial\Omega$ for some constant $c \in \mathbb{R}$. Then the first eigenvalue $p_1$ of problem~\eqref{problem: Fourth} with $\varphi = 0$ satisfies
	\[
	p_1 \ge 2c.
	\]
	Moreover, equality holds if and only if $\Omega$ is a disk of radius $\frac{1}{c}$.
\end{thm}

Subsequently, Wang--Xia~\cite{WX} extended Theorem~\ref{thm: Payne fourth} to compact manifolds with boundary and nonnegative Ricci curvature (see also~\cite{FGW} for the case of higher-dimensional Euclidean domains) with $\varphi = 0$. More recently, Batista--Lima--Sousa--Vieira~\cite{Batista2023} obtained a lower bound for the first eigenvalue of the fourth-order Steklov problem on free boundary minimal hypersurfaces in a Euclidean ball. In Section~3, we establish the same sharp lower bound for the first eigenvalue of problem~\eqref{problem: Fourth} when $\Omega$ is a compact surface with nonnegative Gaussian curvature and the function $\varphi(t)$ satisfies {{$t\varphi(t) \le 0$}} (see Theorem~\ref{thm: main2}).

\section{First eigenvalue of the second-order Steklov-type eigenvalue problem}

The first eigenvalue $\sigma_1$ for the Steklov-type eigenvalue problem~\eqref{problem: Second} is characterized as follows:
\[
\sigma_1 = \inf\left\{ \frac{\int_\Omega |\nabla f|^2 + \int_\Omega f\varphi(f)}{\int_{\partial \Omega} f^2} : f \in H^1(\Omega) \setminus \{0\}, \, f\varphi(f) \in L^1(\Omega) \right\}.
\]
In this section, we establish a sharp lower bound for the first eigenvalue of problem~\eqref{problem: Second} on a compact surface under appropriate curvature conditions. More precisely, we prove the following:

\begin{thm} \label{thm: main1}
	Let $\Omega$ be a compact surface with smooth boundary $\partial\Omega$. Let $\varphi : \mathbb{R} \to \mathbb{R}$ be a smooth function satisfying condition~\eqref{condition on phi}. Assume that the geodesic curvature satisfies $k_g \ge {{c > 0}}$ along $\partial\Omega$ for some constant $c \in \mathbb{R}$, and that the Gaussian curvature $K$ satisfies
	\[
	K(x) \ge -\varphi'(u(x)) \quad \text{for all } x \in \Omega,
	\]
	where $u(x)$ denotes the first eigenfunction of problem~\eqref{problem: Second} with corresponding eigenvalue $\sigma_1$. Then we have
	\[
	\sigma_1 + \frac{1}{\sigma_1} \max_{\partial\Omega} \frac{\varphi(u(x))}{u(x)} \ge c.
	\]
	Moreover, equality holds if and only if $\Omega$ is a Euclidean disk of radius $\frac{1}{c}$ and $\varphi \equiv 0$ on~$\Omega$.
\end{thm}

\begin{proof}
The Bochner formula implies that
\begin{align} \label{ineq: Bochner}
\frac12\,\Delta|\nabla u|^2
  &= \lvert\mathrm{Hess}\,u\rvert^2     + \langle \nabla u,\nabla\varphi(u)\rangle     + \mathrm{Ric}(\nabla u,\nabla u) \nonumber \\
  &= \lvert\mathrm{Hess}\,u\rvert^2+\varphi'(u)\,\lvert\nabla u\rvert^2+K\,\lvert\nabla u\rvert^2   \nonumber \\
  &\ge0,
\end{align}
which shows that $|\nabla u|^2$ is a subharmonic function in~$\Omega$.  By the strong maximum principle, we have two cases:\\

\noindent {\bf Case I:} The maximum value of $|\nabla u|^2$ on $\overline{\Omega}$ occurs only on the boundary $\partial \Omega$.\\
\noindent {\bf Case II:} The function $|\nabla u|^2$ is constant on $\Omega$.\\

In Case I, the maximum value of $|\nabla u|^2$ is attained at some point $p\in\partial\Omega$. Choose a local orthonormal frame $\{e_1,e_2\}$ near $p$ satisfying that $e_1$ is tangent to $\partial\Omega$ and $e_2=\nu$.  Note that
\[
\nabla_{e_1}e_1 \;=\; -\,k_g\,e_2,
\]
where $k_g$ denotes the geodesic curvature of $\partial\Omega$. Note that, at $p\in\partial\Omega$,
\begin{align*}
0 &=\;e_1\bigl\langle \nabla u,\nabla u\bigr\rangle \\
&=\;2\,\mathrm{Hess}\,u\,(e_1,\nabla u)\\
&=\;2\Bigl(\langle\nabla u,e_1\rangle\,\mathrm{Hess}\,u(e_1,e_1) + \langle\nabla u,e_2\rangle\,\mathrm{Hess}\,u(e_1,e_2)\Bigr).
\end{align*}
Moreover,
\begin{align} \label{eq: Hess}
\mathrm{Hess}\,u(e_1,e_2) & = e_1 e_2 u - (\nabla_{e_1}e_2)u \nonumber \\
&= e_1(\sigma_1\,u) - k_g \,e_1u \nonumber \\
&= (\sigma_1-k_g)\, \langle\nabla u, e_1\bigr\rangle.
\end{align}
Thus we get
\[
\bigl\langle\nabla u,e_1\bigr\rangle
\;\Bigl(
   \mathrm{Hess}\,u(e_1,e_1)
   \;+\;
   (\sigma_1-k_g)\,\langle\nabla u,e_2\rangle
\Bigr)
=0,
\]
which gives two possibilities at $p$: \\

\noindent  \text{(i)} $\langle\nabla u,e_1\rangle=0$ \\
 \text{(ii)} $\mathrm{Hess}\,u(e_1,e_1)=(k_g-\sigma_1)\,\langle\nabla u,e_2\rangle.$ \\

\noindent For (i), suppose $\langle\nabla u,e_1\rangle=0$ at $p \in \partial \Omega$.  By the above, we see that
\[
\mathrm{Hess}\,u(e_1,e_2)=0
\quad\text{at }p.
\]
Since $p$ is the maximum point of the function $|\nabla u|^2$, we have
\[
\mathrm{Hess}\,|\nabla u|^2(e_1,e_1)\le0
\quad\text{at }p.
\]
Thus we get
\begin{align*}
0&\ge \mathrm{Hess}\,|\nabla u|^2(e_1,e_1)\\
&= e_1e_1\langle\nabla u,\nabla u\rangle
  -(\nabla_{e_1}e_1)\langle\nabla u,\nabla u\rangle\\
&= 2\,e_1\bigl(\mathrm{Hess}\,u(e_1,\nabla u)\bigr)
  +k_g\,e_2\langle\nabla u,\nabla u\rangle\\
&= 2\,e_1\bigl(\langle\nabla u,e_1\rangle\,\mathrm{Hess}\,u(e_1,e_1)
             +\langle\nabla u,e_2\rangle\,\mathrm{Hess}\,u(e_1,e_2)\bigr)
  +2\,k_g\,\mathrm{Hess}\,u(e_2,\nabla u)\\
&= 2\bigl[\mathrm{Hess}\,u(e_1,e_1)^2
         +\langle\nabla u,\nabla_{e_1}e_1\rangle\,\mathrm{Hess}\,u(e_1,e_1)\bigr]\\
&\quad+2\,\langle\nabla u,e_2\rangle\,e_1\bigl(\mathrm{Hess}\,u(e_1,e_2)\bigr)
  +2\,k_g\,\langle\nabla u,e_2\rangle\,\mathrm{Hess}\,u(e_2,e_2),
\end{align*}
where we used the fact that $\langle\nabla u,e_1\rangle=0$ and $\mathrm{Hess}\,u(e_1,e_2)=0$ at $p \in \partial \Omega$. Moreover, since $\langle\nabla u,e_2\rangle=\sigma_1\,u$ by the boundary condition, we have
\[
e_1\langle\nabla u,e_2\rangle
=\sigma_1\langle\nabla u,e_1\rangle,
\]
which shows that
\[
\mathrm{Hess}\,u(e_1,e_2)
+\langle\nabla u,\nabla_{e_1}e_2\rangle
=\sigma_1\langle\nabla u,e_1\rangle.
\]
Thus
\[
\mathrm{Hess}\,u(e_1,e_2)
=(\sigma_1-k_g)\,\langle\nabla u,e_1\rangle.
\]
Differentiating both sides in the direction of $e_1$ at $p \in \partial \Omega$ gives
\begin{align*}
e_1\bigl(\mathrm{Hess}\,u(e_1,e_2)\bigr) &=e_1\bigl((\sigma_1-k_g)\langle\nabla u,e_1\rangle\bigr) \\
&=(\sigma_1-k_g)\bigl(\mathrm{Hess}\,u(e_1,e_1) +\langle\nabla u,\nabla_{e_1}e_1\rangle\bigr).
\end{align*}
Using this, we obtain
\begin{align} \label{ineq: bbb}
0 &\ge \mathrm{Hess}\,u(e_1,e_1)^2 -k_g\langle\nabla u,e_2\rangle\,\mathrm{Hess}\,u(e_1,e_1) \nonumber \\
&\quad +\langle\nabla u,e_2\rangle(\sigma_1-k_g)\bigl(\mathrm{Hess}\,u(e_1,e_1)-k_g\langle\nabla u,e_2\rangle\bigr) \nonumber \\
&\quad + k_g\langle\nabla u,e_2\rangle\,\mathrm{Hess}\,u(e_2,e_2)
\end{align}
at $p \in \partial \Omega$. On the other hand, the Hopf boundary point lemma implies that
\begin{align} \label{ineq: aaa}
0 &< e_2\langle\nabla u,\nabla u\rangle \nonumber \\
&=2\,\mathrm{Hess}\,u(\nabla u, e_2) \nonumber \\
&=2\, \langle\nabla u,e_2\rangle\,\mathrm{Hess}\,u(e_2,e_2) \nonumber \\
&=\langle\nabla u,e_2\rangle\bigl(\varphi(u)-\mathrm{Hess}\,u(e_1,e_1)\bigr)
\end{align}
at $p \in \partial \Omega$. Combining (\ref{ineq: bbb}) and (\ref{ineq: aaa}), we get
\begin{align*}
&\mathrm{Hess}\,u(e_1,e_1)^2 +2(\sigma_1-k_g)\,\langle\nabla u,e_2\rangle\,\mathrm{Hess}\,u(e_1,e_1) -k_g(\sigma_1-k_g)\,\langle\nabla u,e_2\rangle^2 \\
&\quad \le \sigma_1 \langle\nabla u,e_2\rangle\,\mathrm{Hess}\,u(e_1,e_1) - k_g\,\langle\nabla u,e_2\rangle\,(\varphi(u) - \mathrm{Hess}\,u(e_1,e_1)) \\
&\quad < k_g\,\langle\nabla u,e_2\rangle\bigl(\mathrm{Hess}\,u(e_1,e_1)-\varphi(u)\bigr) + \sigma_1\,\varphi(u) \langle\nabla u,e_2\rangle\,\\
&\quad < \sigma_1\,\varphi(u)\,\langle\nabla u,e_2\rangle
\end{align*}
at $p \in \partial \Omega$.  We claim that $u(p)\neq 0$. To see this, suppose that $u(p)=0$. Then
\begin{align} \label{eq: u(p) is not zero}
\langle\nabla u,e_2\rangle = \frac{\partial u}{\partial \nu}= \sigma_1 u= 0
\end{align}
at $p \in \partial \Omega$, which is a contradiction by (\ref{ineq: aaa}). Using this, we finally obtain
\begin{align*}
0 &\le \left(\mathrm{Hess}\,u(e_1,e_1)+(\sigma_1-k_g)\langle\nabla u,e_2\rangle\right)^2 \\
&<(\sigma_1-k_g)^2\langle\nabla u,e_2\rangle^2 + k_g(\sigma_1-k_g)\langle\nabla u,e_2\rangle^2 + \sigma_1\,\varphi(u)\,\langle\nabla u,e_2\rangle\\
&= \sigma_1(\sigma_1-k_g)\langle\nabla u,e_2\rangle^2 + \sigma_1\,u\,\frac{\varphi(u)}{u}\,\langle\nabla u,e_2\rangle \\
&= \sigma_1(\sigma_1-k_g)\langle\nabla u,e_2\rangle^2 + \frac{\varphi(u)}{u}\,\langle\nabla u,e_2\rangle^2 \\
&= \left(\sigma_1^2 - k_g\,\sigma_1 + \frac{\varphi(u)}{u}\right) \langle\nabla u,e_2\rangle^2
\end{align*}
at $p \in \partial \Omega$. Therefore we have
\[
\sigma_1^2 - k_g\,\sigma_1 + \frac{\varphi(u)}{u} > 0,
\]
which gives the conclusion.\\

For (ii), suppose that
\begin{align} \label{eq: hhh}
\mathrm{Hess}\,u(e_1,e_1) = (k_g - \sigma_1)\,\langle\nabla u,e_2\rangle
\quad\text{at }p.
\end{align}
The Hopf boundary point lemma implies that
\[
0 < \frac{\partial |\nabla u|^2}{\partial\nu}(p),
\]
which gives
\begin{align} \label{ineq: iii}
0 < e_2\langle\nabla u,\nabla u\rangle
\quad\text{at }p.
\end{align}
Combining (\ref{eq: hhh}) and (\ref{ineq: iii}) gives
\begin{align} \label{ineq: ccc}
0 &< \mathrm{Hess}\,u(e_2,\nabla u)  \nonumber \\
  &= \langle\nabla u,e_1\rangle\,\mathrm{Hess}\,u(e_1,e_2)    + \langle\nabla u,e_2\rangle\,\mathrm{Hess}\,u(e_2,e_2) \nonumber \\
  &= \langle\nabla u,e_1\rangle\,\mathrm{Hess}\,u(e_1,e_2)    + \langle\nabla u,e_2\rangle\bigl(\varphi(u)-\mathrm{Hess}\,u(e_1,e_1)\bigr)  \nonumber \\
  &= (\sigma_1-k_g)\,\langle\nabla u,e_1\rangle^2    + \langle\nabla u,e_2\rangle\bigl(\varphi(u)-\mathrm{Hess}\,u(e_1,e_1)\bigr)  \nonumber \\
  &= \varphi(u)\,\langle\nabla u,e_2\rangle    + (\sigma_1-k_g)\,\lvert\nabla u\rvert^2
\end{align}
at $p \in \partial \Omega$. If $u(p)=0$, then we get
$$\sigma_1>k_g (p)$$
by (\ref{eq: u(p) is not zero}) and (\ref{ineq: ccc}). If $u(p) \neq 0$, then (\ref{ineq: ccc}) shows that
\begin{align*}
(\sigma_1-k_g)\,\lvert\nabla u\rvert^2 &> -\varphi(u)\,\langle\nabla u,e_2\rangle \\
&= -\,\frac{1}{\sigma_1}\,\frac{\varphi(u)}{u}\,\sigma_1\,u\,\langle\nabla u,e_2\rangle \\
&= -\,\frac{1}{\sigma_1}\,\frac{\varphi(u)}{u}\,\langle\nabla u,e_2\rangle^2 \\
&\ge\;-\,\frac{1}{\sigma_1}\,\frac{\varphi(u)}{u}\,\lvert\nabla u\rvert^2
\end{align*}
at $p \in \partial \Omega$. Here we used the condition on $\varphi$ that $\dfrac{\varphi (t)}{t} \ge 0$. Therefore
\[
\sigma_1 - k_g + \frac{1}{\sigma_1} \frac{\varphi(u)}{u} > 0
\quad\text{at }p,
\]
which gives the conclusion. \\

In Case II, $|\nabla u|^2$ is constant in $\Omega$. From (\ref{ineq: Bochner}), it follows that
$$\mathrm{Hess}\,u=0 ~~{\rm and }~~ \varphi ' (u) = - K.$$
Since $\Delta u = \varphi (u)=0$ in $\Omega$, we see that
$$K=0~~ {\rm on}~~ \Omega,$$
i.e., $\Omega$ is flat. Moreover, along the boundary $\partial \Omega$, we have
$$0= (\sigma_1 -k_g) \langle \nabla u, e_1\rangle$$
by (\ref{eq: Hess}). Suppose that $\langle \nabla u, e_1\rangle = 0$. From the fact that $\nabla u$ is constant and the boundary condition, it follows that $u$ is constant, which is a contradiction. Hence
$$\sigma_1 -k_g = 0$$
along $\partial \Omega$, which implies that $\Omega$ is a flat disk of radius $\frac{1}{\sigma_1}$. This completes the proof.

\end{proof}

In particular, if we choose the function $\varphi(t) = \alpha t$ for some constant $\alpha \ge 0$, then the eigenvalue problem~\eqref{problem: Second} becomes a {{Schr\"odinger--Steklov}} eigenvalue problem (see~\cite{Auchmuty, Mavinga} for instance). In this case, one immediately obtains the following consequence.

\begin{cor} \label{cor: Schrodinger}
	Let $\Omega$ be a compact surface with smooth boundary $\partial\Omega$ and with Gaussian curvature satisfying $K \ge -\alpha$ for some constant $\alpha \ge 0$. Define the smooth function $\varphi : \mathbb{R} \to \mathbb{R}$ by $\varphi(t) = \alpha t$. Assume that the geodesic curvature satisfies $k_g \ge c > 0$ along $\partial\Omega$ for some constant $c \in \mathbb{R}$. Then the first eigenvalue $\sigma_1$ of problem~\eqref{problem: Second} satisfies
	\[
	\sigma_1 + \frac{\alpha}{\sigma_1} \ge c.
	\]
	Moreover, equality holds if and only if $\Omega$ is a Euclidean disk of radius $1/c$ and $\alpha = 0$.
\end{cor}

\noindent We remark that when $\alpha = 0$, Corollary~\ref{cor: Schrodinger} recovers the results by Payne~\cite{Payne1970} and Escobar~\cite{Escobar1997}.

\section{First eigenvalue of the fourth-order Steklov-type eigenvalue problem}

In this section, we consider the fourth-order eigenvalue problem~\eqref{problem: Fourth} introduced in Section~1. The first eigenvalue $p_1$ of problem~\eqref{problem: Fourth} is variationally characterized by
\[
p_1 = \inf \left\{ \frac{\int_\Omega (\Delta f)^2 - \int_\Omega f \varphi(f)}{\int_{\partial \Omega} \left( \frac{\partial f}{\partial \nu} \right)^2} : f \in H^2(\Omega) \cap H^1_0(\Omega),\, f \varphi(f) \in L^1(\Omega),\, 0 \neq \frac{\partial f}{\partial \nu} \in L^2(\partial \Omega) \right\}.
\]
We establish a sharp lower bound for $p_1$ as follows:

\begin{thm} \label{thm: main2}
	Let $\Omega$ be a compact surface with smooth boundary and  Gaussian curvature $K\ge 0$. Let $\varphi: \mathbb{R} \to \mathbb{R}$ be a smooth function satisfying {{$t \varphi(t) \le 0$ for all $t \in \mathbb{R}$}}. Assume that the geodesic curvature satisfies $k_g \ge {{c > 0}}$ along $\partial\Omega$ for some constant $c \in \mathbb{R}$. Then the first eigenvalue $p_1$ of problem~\eqref{problem: Fourth} satisfies
	\[
	p_1 \ge 2c.
	\]
	Moreover, equality holds if and only if $\Omega$ is a Euclidean disk of radius $\frac{1}{c}$.
\end{thm}

\begin{proof}
Define the auxiliary function $w$ by
	\[
	w := |\nabla u|^2 - u \Delta u.
	\]
The Bochner formula with our assumption shows that
\begin{align} \label{ineq: Bochner 2}
\frac{1}{2} \Delta w
&= \frac{1}{2} \Delta |\nabla u|^2 - \frac{1}{2} \Delta(u\Delta u) \nonumber \\
&= |\mathrm{Hess} \, u|^2 + \langle \nabla u, \nabla (\Delta U) \rangle + K|\nabla u|^2 \nonumber \\
&\quad - \frac{1}{2} \left[ u \Delta^2 u + (\Delta u)^2 + 2 \langle \nabla u, \nabla (\Delta u) \rangle \right] \nonumber \\
&= |\mathrm{Hess} \, u|^2 - \frac{1}{2} (\Delta u)^2 + K|\nabla u|^2 - \frac{1}{2} u \varphi(u)\nonumber \\
&\ge 0,
\end{align}
where we used the Schwarz inequality $|\mathrm{Hess}\, u|^2 \ge \frac{1}{2} (\Delta u)^2$ in dimension $2$. Thus it follows that $w$ is subharmonic in $\Omega$. By the strong maximum principle, there are two cases:\\

\noindent {\bf Case I:} The maximum value of $w$ on $\overline{\Omega}$ occurs only on the boundary $\partial \Omega$.\\
\noindent {\bf Case II:} The function $w$ is constant in $\Omega$.\\

For Case I, suppose that the maximum value of $w$ is attained at some point $p\in\partial\Omega$. As before, we choose a local orthonormal frame $\{e_1,e_2\}$ near $p$ satisfying that $e_1$ is tangent to $\partial\Omega$ and $e_2=\nu$ with
\[
\nabla_{e_1}e_1 \;=\; -\,k_g\,e_2,
\]
where $k_g$ denotes the geodesic curvature of $\partial\Omega$. At $p\in \partial \Omega$,
\begin{align} \label{eq: normal derivative of w}
\frac{\partial w}{\partial \nu}(p) &= e_2 w = e_2 \langle \nabla u, \nabla u \rangle - e_2(u \Delta u) \nonumber \\
&= 2 \operatorname{Hess} u(\nabla u, e_2) - \langle \nabla u, e_2 \rangle \Delta u.
\end{align}
Since $u$ vanishes along $\partial \Omega$, we have $\langle \nabla u, e_1 \rangle = 0$ on $\partial \Omega$. Thus
\begin{align} \label{eq: Hess u e1e1}
\mathrm{Hess}\,u(e_1,e_1) & = e_1 e_1 u - (\nabla_{e_1}e_1)u \nonumber \\
&= e_1 \langle \nabla u, e_1 \rangle + k_g \langle \nabla u, e_2 \rangle \nonumber \\
&= k_g \langle \nabla u, e_2 \rangle
\end{align}
at $p\in \partial \Omega$. Combining (\ref{eq: normal derivative of w}),  (\ref{eq: Hess u e1e1}) and the boundary condition gives
\begin{align}  \label{eq: sss}
\frac{\partial w}{\partial \nu}(p) &= 2 \langle \nabla u, e_2 \rangle \operatorname{Hess} u(e_2, e_2) - p_1 \langle \nabla u, e_2 \rangle^2
\end{align}
Observe that, by (\ref{eq: Hess u e1e1}),
\begin{align} \label{eq: ggg}
\mathrm{Hess}\,u(e_2,e_2) & = \Delta u - \mathrm{Hess}\,u(e_1,e_1)  \nonumber \\
&= (p_1-k_g) \langle \nabla u, e_2 \rangle
\end{align}
at $p\in \partial \Omega$. Thus, from (\ref{eq: sss}) and (\ref{eq: ggg}), it follows that
\begin{align*}
\frac{\partial w}{\partial \nu}(p) &=  \langle \nabla u, e_2 \rangle^2 (p_1 - 2 k_g).
\end{align*}
Since $\frac{\partial w}{\partial \nu}(p)>0$ by the Hopf boundary point lemma and $\langle \nabla u, e_2 \rangle \neq 0$, we obtain
$$p_1 > 2 k_g,$$
which shows that $p_1 > 2 c.$

For Case II, suppose that $w\equiv constant$ in $\Omega$. By (\ref{ineq: Bochner 2}), we see that
$$K\equiv 0~ {\mbox{ and }}~ \varphi \equiv 0 ~{\mbox{ in }} \Omega.$$
Note that $u = 0$ on $\partial \Omega$. From this, it follows that
\begin{align*}
w= \left(\frac{\partial u}{\partial \nu} \right)^2 = {\mbox{constant}}~{\mbox{ on }} \partial \Omega,
\end{align*}
which implies that
\begin{align*}
\Delta u= p_1 \frac{\partial u}{\partial \nu} = {\mbox{constant}}~{\mbox{ on }} \partial \Omega.
\end{align*}
Since $\Delta u$ is harmonic in $\Omega$, we have the following overdetermined boundary value problem:

\begin{align*}
\begin{cases}
	\Delta u = p_1c_1 & \text{in } \Omega, \\
	u = 0 & \text{on } \partial \Omega, \\
	\frac{\partial u}{\partial \nu} = c_1   & \text{on } \partial \Omega
\end{cases}
\end{align*}
for some constant $c_1$. By the well-known result due to Serrin \cite{Serrin}, $\Omega$ is a Euclidean disk of radius $r$. One can check that
$$p_1 = \frac{2}{r} =2k_g \ge 2c,$$
which completes the proof.
\end{proof}

\vskip 0.3cm
\noindent
\noindent{\bf Acknowledgment: } The second author was supported by the National Research Foundation of Korea (RS-2021-NR058050). Part of this work was carried out while the first author was a postdoctoral research fellow and the second author was visiting at the University of California, Santa Barbara, in 2024. They wish to express their gratitude to Professor Guofang Wei for her warm hospitality during their stay.


\vskip 1cm
\noindent Gunhee Cho\\
Department of Mathematics\\
Texas State University\\
601 University Drive, San Marcos, TX 78666.\\
{\tt e-mail:wvx17@txstate.edu}\\
URL:https://sites.google.com/view/enjoyingmath/

\bigskip
\noindent Keomkyo Seo\\
Department of Mathematics and Research Institute of Natural Sciences\\
Sookmyung Women's University\\
Cheongpa-ro 47-gil 100, Yongsan-ku, Seoul, 04310, Korea \\
{\tt E-mail:kseo@sookmyung.ac.kr}\\
URL: http://sites.google.com/site/keomkyo/


\begin{thebibliography}{99}
	
	\bibitem{AM} G. Alessandrini, R. Magnanini, {\em Symmetry and nonsymmetry for the overdetermined Stekloff eigenvalue problem}, Z. Angew. Math. Phys. {\bf 45} (1994), no. 1, 44--52.
	
	\bibitem{Auchmuty} G. Auchmuty, {\em Steklov eigenproblems and the representation of solutions of elliptic boundary value problems}, Numer. Funct. Anal. Optim. {\bf 25} (2004), no. 3--4, 321--348.
	
	\bibitem{Batista2023} R. Batista, B. Lima, P. Sousa, B. Vieira, {\em Estimate for the first fourth Steklov eigenvalue of a minimal hypersurface with free boundary}, Pacific J. Math. {\bf 325} (2023), no. 1, 1--17.
	
	\bibitem{BGM} E. Berchio, F. Gazzola, E. Mitidieri, {\em Positivity preserving property for a class of biharmonic elliptic problems}, J. Differential Equations {\bf 229} (2006), no. 1, 1--23.
	
	\bibitem{BFG} D. Bucur, A. Ferrero, F. Gazzola, {\em On the first eigenvalue of a fourth order Steklov problem}, Calc. Var. Partial Differential Equations {\bf 35} (2009), no. 1, 103--131.
	
	\bibitem{CGGS} B. Colbois, A. Girouard, C. Gordon, D. Sher, {\em Some recent developments on the Steklov eigenvalue problem}, Rev. Mat. Complut. (2023). \url{https://doi.org/10.1007/s13163-023-00480-3}
	
	\bibitem{DK} J. A. J. Duncan, A. Kumar, {\em The first Steklov eigenvalue on manifolds with non-negative Ricci curvature and convex boundary}, J. Geom. Anal. {\bf 35} (2025), no. 3, Paper No. 95, 18~pp.
	
	\bibitem{Escobar1997} J. F. Escobar, {\em The geometry of the first non-zero Stekloff eigenvalue}, J. Funct. Anal. {\bf 150} (1997), no. 2, 544--556.
	
	\bibitem{Escobar1999} J. F. Escobar, {\em An isoperimetric inequality and the first Steklov eigenvalue}, J. Funct. Anal. {\bf 165} (1999), no. 1, 101--116.
	
	\bibitem{Escobar2000} J. F. Escobar, {\em A comparison theorem for the first non-zero Steklov eigenvalue}, J. Funct. Anal. {\bf 178} (2000), no. 1, 143--155.
	
	\bibitem{FGW} A. Ferrero, F. Gazzola, T. Weth, {\em On a fourth order Steklov eigenvalue problem}, Analysis (Munich) {\bf 25} (2005), no. 4, 315--332.
	
	\bibitem{FS2011} A. Fraser, R. Schoen, {\em The first Steklov eigenvalue, conformal geometry, and minimal surfaces}, Adv. Math. {\bf 226} (2011), no. 5, 4011--4030.
	
	\bibitem{FS2016} A. Fraser, R. Schoen, {\em Sharp eigenvalue bounds and minimal surfaces in the ball}, Invent. Math. {\bf 203} (2016), no. 3, 823--890.
	
	\bibitem{FS2020} A. Fraser, R. Schoen, {\em Some results on higher eigenvalue optimization}, Calc. Var. Partial Differential Equations {\bf 59} (2020), no. 5, Paper No. 151, 22~pp.
	
	\bibitem{GP} A. Girouard, I. Polterovich, {\em Spectral geometry of the Steklov problem}, J. Spectr. Theory {\bf 7} (2017), no. 2, 321--359.
	
	\bibitem{KS} J. R. Kuttler, V. G. Sigillito, {\em Inequalities for membrane and Stekloff eigenvalues}, J. Math. Anal. Appl. {\bf 23} (1968), 148--160.
	
	\bibitem{LS} E. Lee, K. Seo, {\em An overdetermined Steklov eigenvalue problem on Riemannian manifolds with nonnegative Ricci curvature}, Results Math. {\bf 80} (2025), no. 4, Paper No. 102.
	
	\bibitem{Mavinga} N. M. Mavinga, {\em Steklov spectrum and elliptic problems with nonlinear boundary conditions}, Notices Amer. Math. Soc. {\bf 70} (2023), no. 2, 214--222.
	
	\bibitem{Payne1970} L. E. Payne, {\em Some isoperimetric inequalities for harmonic functions}, SIAM J. Math. Anal. {\bf 1} (1970), no. 3, 354--359.
	
	\bibitem{PP} L. E. Payne, G. A. Philippin, {\em Some overdetermined boundary value problems for harmonic functions}, Z. Angew. Math. Phys. {\bf 42} (1991), no. 6, 864--873.
	
	\bibitem{RS} S. Raulot, A. Savo, {\em Sharp bounds for the first eigenvalue of a fourth-order Steklov problem}, J. Geom. Anal. {\bf 25} (2015), no. 3, 1602--1619.
	
	\bibitem{Serrin} J. Serrin, {\em A symmetry problem in potential theory}, Arch. Ration. Mech. Anal. {\bf 43} (1971), no. 4, 304--318.
	
	\bibitem{WX} Q. Wang, C. Xia, {\em Sharp bounds for the first non-zero Stekloff eigenvalues}, J. Funct. Anal. {\bf 257} (2009), no. 8, 2635--2644.
	
	\bibitem{Weinstock} R. Weinstock, {\em Inequalities for a classical eigenvalue problem}, J. Rational Mech. Anal. {\bf 3} (1954), 745--753.
	
	\bibitem{XX} C. Xia, C. Xiong, {\em Escobar's conjecture on a sharp lower bound for the first nonzero Steklov eigenvalue}, Peking Math. J. {\bf 7} (2024), no. 2, 759--778.
	
\end{thebibliography}
\end{document}